\documentclass[12pt]{amsart}%
\usepackage{amscd,amsmath,latexsym,amsthm,amsfonts,amssymb,graphicx,color,geometry,hyperref}
\usepackage[utf8]{inputenc}
\usepackage{amsmath}
\usepackage{amsfonts}
\usepackage{amssymb}
\usepackage{graphicx}%
\providecommand{\U}[1]{\protect\rule{.1in}{.1in}}
\providecommand{\U}[1]{\protect\rule{.1in}{.1in}}

\newtheorem{theorem}{Theorem}[section]

\newtheorem{example}[theorem]{Example}

\newtheorem{lemma}[theorem]{Lemma}

\newtheorem{definition}[theorem]{Definition}
\numberwithin{equation}{section}

\begin{document}
\title[Lineability and spaceability: a new approach]{Lineability and spaceability: a new approach}
\author[F\'{a}varo]{V.V. F\'{a}varo}
\address{Faculdade de Matem\'{a}tica \\
\indent Universidade Federal de Uberl\^andia \\
\indent
	38.400-902 - Uberl\^andia}
\email{vvfavaro@gmail.com}
\author[Pellegrino]{D. Pellegrino}
\address{Departamento de Matem\'{a}tica \\
\indent
	Universidade Federal da Para\'{\i}ba \\
\indent
	58.051-900 - Jo\~{a}o Pessoa, Brazil}
\email{pellegrino@pq.cnpq.br}
\author[Tomaz]{D. Tomaz}
\address{Departamento de Matem\'{a}tica \\
\indent
	Universidade Federal da Para\'{\i}ba \\
\indent
	58.051-900 - Jo\~{a}o Pessoa, Brazil}
\email{danieltomazmatufpb@gmail.com}
\thanks{Daniel Tomaz is supported by Capes}
\thanks{Daniel Pellegrino is supported by CNPq}
\thanks{Vinicius F\'{a}varo is supported by FAPEMIG Grant APQ-03181-16; and CNPq Grant 310500/2017-6}
\thanks{2010 Mathematics Subject Classification: 15A03, 46A16; 46A45}
\keywords{Cardinal numbers, lineability, spaceability}

\begin{abstract}
The area of research called \textquotedblleft Lineability\textquotedblright%
\ looks for linear structures inside exotic subsets of vector spaces. In the
last decade lineability/spaceability has been investigated in rather general
settings; for instance, Set Theory, Probability Theory, Functional Analysis,
Measure Theory, etc. It is a common feeling that positive results on
lineability/spaceability are quite natural (i.e., in general \textquotedblleft
large\textquotedblright subspaces can be found inside exotic subsets of vector
spaces, in quite different settings) and more restrictive approaches have been
persecuted. In this paper we introduce and explore a new approach in this direction.

\end{abstract}
\maketitle



\section{Introduction and preliminaries}

The notions of lineability and spaceability were introduced by V. Gurariy and
L. Quarta in \cite{Quarta} and by Aron, Gurariy, and Seoane-Sep\'{u}lveda in
\cite{Aron}. For a comprehensive background on lineability and spaceability we
recommend the recent monograph \cite{book}. This line of research investigates
the presence of \textquotedblleft large\textquotedblright\ linear subspaces in
certain mathematical objects with \textit{a priori} no linear structure. The
properties of lineability and spaceability are studied in several contexts
with interesting applications in different fields as norm-attaining operators,
multilinear forms, homogeneous polynomials, sequence spaces, holomorphic
mappings, absolutely summing operators, Peano curves, fractals, among others.
See, for instance, \cite{Nacib, BBFP_PAMS, Bernal, Pellegrino2, Pellegrino1,
BF_MMJ, CFS_PAMS, cariellojfa, Teixeira}, and the references therein.\newline
Let $E$ be a vector space and $\alpha$ be a cardinal number. A subset $A$ of
$E$ is called $\alpha\text{-lineable}$ if $A\cup\left\{  0\right\}  $ contains
an $\alpha$-dimensional linear subspace of $E;$ if $E$ is a topological vector
space, it is called $\alpha$-spaceable if $A\cup\left\{  0\right\}  $ contains
a closed $\alpha$-dimensional linear subspace of $E$. It is well known that,
in general, positive results of lineability are rather usual and the feeling
that \textquotedblleft everything is lineable\textquotedblright\ is somewhat
common. So, a more restrictive notion of lineability is in order. Our paper
investigates a stronger notion of lineability/spaceability, which is rather
more restrictive.\newline Let us establish some notations that will be carried
out along this work. From now on all vector spaces are considered over a fixed
scalar field $\mathbb{K}$ which can be either $\mathbb{R}$ or $\mathbb{C}$ .
We shall denote by $\mathfrak{c}=card\left(  \mathbb{R}\right)  $ and
$\aleph_{0}=card\left(  \mathbb{N}\right)  $. The following concepts are more
restrictive notions of lineability/spaceability inspired by some ideas from
\cite{tonydaniel}:

\begin{definition}
Let $\alpha,\beta,\lambda$ be cardinal numbers and $V$ be a vector space, with
$\dim V=\lambda$ and $\alpha<\beta\leq\lambda$. A set $A\subset V$ is:

(i) $\left(  \alpha,\beta\right)  $-lineable if it is $\alpha$-lineable and
for every subspace $W_{\alpha}\subset V$ with $W_{\alpha}\subset A\cup\left\{
0\right\}  $ and $\dim W_{\alpha}=\alpha$, there is a subspace $W_{\beta
}\subset V$ with $\dim W_{\beta}=\beta$ and $W_{\alpha}\subset W_{\beta
}\subset A\cup\left\{  0\right\}  $.

(ii) $\left(  \alpha,\beta\right)  $-spaceable if it is $\alpha$-lineable and
for every subspace $W_{\alpha}\subset V$ with $W_{\alpha}\subset A\cup\left\{
0\right\}  $ and $\dim W_{\alpha}=\alpha$, there is a closed subspace
$W_{\beta}\subset V$ with $\dim W_{\beta}=\beta$ and $W_{\alpha}\subset
W_{\beta}\subset A\cup\left\{  0\right\}  $. \newline
\end{definition}

The original notion of lineability (or spaceability) is just the case
$\alpha=0$. Let $\alpha_{1},\alpha_{2}$ be cardinal numbers with $\alpha
_{1}<\alpha_{2}\leq\beta$. We start by showing that $\beta$-lineability does
not imply $\left(  \alpha,\beta\right)  $-lineability and \textit{vice versa}.
To provide examples, let us recall that for $p>0$, $\ell_{p}$ denotes the
Banach space ($p$-Banach space if $p<1$) of the sequences $\left(
x_{j}\right)  _{j=1}^{\infty}$ such that
\[
\left\Vert \left(  x_{j}\right)  _{j=1}^{\infty}\right\Vert _{p}=\left(
\sum\limits_{j=1}^{\infty}\left\vert x_{j}\right\vert ^{p}\right)
^{1/p}<\infty.
\]
We denote by $e_{j}$ the canonical vector $\left(  0,...,0,1,0,0,...\right)  $
with $1$ in the $j$-th coordinate, for each $j\in\mathbb{N}$.

More generally, let $p\in[1,\infty)$ and $\Gamma$ be an abstract nonempty set. We denote by $\ell_p(\Gamma)$ the vector space of all functions $f\colon\Gamma\longrightarrow\mathbb{K}$ such that $\sum_{\gamma\in\Gamma}\vert f(\gamma)\vert^p<\infty,$ which becomes a Banach space with the norm 
$$\Vert f\Vert_p= \left(\sum_{\gamma\in \Gamma}\vert f(\gamma)\vert^p\right)^{1/p},
$$
where the sum is defined by
$$\sum_{\gamma\in\Gamma}\vert f(\gamma)\vert^p=\sup\left\{ \sum_{\gamma\in F}\vert f(\gamma)\vert^p: F \textrm{ is a finite subset of } \Gamma\right\}.
$$
We also denote the elements of the canonical generalized Schauder basis of $\ell_p(\Gamma)$ by $e_i$, $i\in\Gamma.$

It is clear that, when $\Gamma=\mathbb{N}$, we have $\ell_p(\mathbb{N})=\ell_p.$ For details about the space $\ell_p(\Gamma)$  and related results we refer to \cite{zizler}.

\begin{example}
Let $n\in\mathbb{N}$ and consider the following subset of $\ell_{p}$:
\[
A=span\left\{  e_{1},e_{2},\ldots,e_{n}\right\}  \cup\left\{  \left(
x_{j}\right)  _{j=1}^{\infty}\in\ell_{p}:x_{1}=x_{2}=\ldots=x_{n}=0\right\}  .
\]
Thus $A$ is $\left(  n+1,\mathfrak{c}\right)  $-lineable, but it is not
$\left(  n,\mathfrak{c}\right)  $-lineable.
\end{example}

\begin{example}
Let $\mathcal{B}=\{e_{i}:i\in\Gamma\}$ be the canonical generalized Schauder
basis in $\ell_{p}\left(  \Gamma\right)  $ with $card\left(  \Gamma\right)
=2^{\aleph_{1}}$. Choose $i_{0},i_{1}\in\Gamma.$ We can easily write%
\[
\mathcal{B}-\{e_{i_{0}},e_{i_{1}}\}=\bigcup\limits_{\left(  \lambda
,\mu\right)  \in\mathbb{R}^{2}}\mathcal{A}_{\left(  \lambda,\mu\right)  }%
\]
as a pairwise disjoint union, with%
\[
card\left(  \mathcal{A}_{\left(  \lambda,\mu\right)  }\right)  =2^{\aleph_{1}}%
\]
for all $\left(  \lambda,\mu\right)  \in\mathbb{R}^{2}$. Let%
\[
A=\bigcup\limits_{\left(  \lambda,\mu\right)  \in\mathbb{R}^{2}}span\left(
\left\{  \lambda e_{i_{0}}+\mu e_{i_{1}}\right\}  \cup\mathcal{A}_{\left(
\lambda,\mu\right)  }\right)  .
\]
Note that there is no vector space $W\subset\ell_{p}\left(  \Gamma\right)  $
with $\dim W=2^{\aleph_{1}}$ and%
\[
span\{e_{i_{0}},e_{i_{1}}\}\subset W\subset A\cup\{0\}.
\]

In fact, if such $W$ exists, since%

\[
e_{i_{0}}\in W\subset A\cup\{0\},
\]
by the very definition of $A$ we conclude that
\begin{equation}
W\subset\bigcup\limits_{\overset{\lambda\in\mathbb{R}}{\lambda\neq0}%
}span\left( \left\{  e_{i_{0}}\right\} \cup\mathcal{A}_{\left(  \lambda
,0\right)  }\right)  \bigcup span\{e_{i_{0}},e_{i_{1}}\}. \label{qwe}%
\end{equation}

Let us prove it: If the inclusion above was false, then there would exist
$w\in W$ with
\[
w\notin\bigcup\limits_{\overset{\lambda\in\mathbb{R}}{\lambda\neq0}%
}span\left(  \left\{  e_{i_{0}}\right\}  \cup\mathcal{A}_{\left(
\lambda,0\right)  }\right)  \bigcup span\{e_{i_{0}},e_{i_{1}}\},
\]
i.e.,
\[
w\in\bigcup\limits_{\underset{\mu\neq0}{\left(  \lambda,\mu\right)
\in\mathbb{R}^{2}}}span\left(  \left\{  \lambda e_{i_{0}}+\mu e_{i_{1}%
}\right\}  \cup\mathcal{A}_{\left(  \lambda,\mu\right)  }\right)  \bigcup
span\left(  \mathcal{A}_{\left(  0,0\right)  }\right)  \text{ and }w\notin
span\{e_{i_{0}},e_{i_{1}}\}.
\]
Thus
\[
w=r\left(  \lambda e_{i_{0}}+\mu e_{i_{1}}\right)  +v
\]
with $v\in span\left(  \mathcal{A}_{\left(  \lambda,\mu\right)  }\right)
,v\neq0,$ and $r\in\mathbb{R}$, where $\mu\neq0$ or $\left(  \lambda
,\mu\right)  =(0,0).$ First suppose that $\mu\neq0$. Since $W$ is a vector
space containing $e_{i_{0}}$ and $e_{i_{1}},$ we have%
\[
w+\left(  -r\lambda+\lambda\right)  e_{i_{0}}+\left(  -r\mu+\mu+s\right)
e_{i_{1}}\in W,
\]
i.e.,%
\[
\lambda e_{i_{0}}+\left(  \mu+s\right)  e_{i_{1}}+v\in W
\]
for any choice of $s\in\mathbb{R}.$ Let us choose $s\neq-\mu$ such that%
\[
\left(  \lambda,\mu\right)  \text{ and }\left(  \lambda,\mu+s\right)  \text{
are linearly independent .}%
\]
Then,
\[
\lambda e_{i_{0}}+\left(  \mu+s\right)  e_{i_{1}}\neq0,
\]
and since $v\in span\left(  \mathcal{A}_{\left(  \lambda,\mu\right)  }\right)
$ we have that
\[
v\notin\bigcup\limits_{\left(  \alpha,\beta\right)  \neq\left(  \lambda
,\mu\right)  }span\left(  \mathcal{A}_{\left(  \alpha,\beta\right)  }\right)
.
\]
Now it follows from the definition of $A$ that%
\[
\lambda e_{i_{0}}+\left(  \mu+s\right)  e_{i_{1}}+v\notin A,
\]
which is a contradiction.

Now, let us consider the case $\left(  \lambda,\mu\right)  =(0,0).$ In this
case
\[
w\in span\left(  \mathcal{A}_{\left(  0,0\right)  }\right)  \text{ and so }
w\notin\bigcup\limits_{\left( \alpha,\beta\right) \neq\left(  0,0\right)
}span\left(  \mathcal{A}_{\left( \alpha,\beta\right)  }\right)  .
\]
Since $e_{i_{0}},e_{i_{1}}\in W$, for any choice of $\left( \alpha
,\beta\right) \neq(0,0)$ we obtain
\[
\alpha e_{i_{0}}+\beta e_{i_{1}}+w\in W.
\]
By the definition of $A$ we have%
\[
\alpha e_{i_{0}}+\beta e_{i_{1}}+w\notin A,
\]
a contradiction.

Analogously, we conclude that
\begin{equation}
W\subset\bigcup\limits_{\mu\in\mathbb{R}\setminus\{0\}}span\left(  \left\{
e_{i_{1}}\right\}  \cup\mathcal{A}_{\left(  0,\mu\right)  }\right)  \bigcup
span\{e_{i_{0}},e_{i_{1}}\}.\label{1212}%
\end{equation}
Thus, from (\ref{qwe}) and (\ref{1212}),
\[
W\subset\left(  \bigcup\limits_{\lambda\in\mathbb{R}\setminus\{0\}}span\left(
\left\{  e_{i_{0}}\right\}  \cup\mathcal{A}_{\left(  \lambda,0\right)
}\right)  \right)  \bigcap\left(  \bigcup\limits_{\mu\in\mathbb{R}%
\setminus\{0\}}span\left(  \left\{  e_{i_{1}}\right\}  \cup\mathcal{A}%
_{\left(  0,\mu\right)  }\right)  \right)  \bigcup span\{e_{i_{0}},e_{i_{1}%
}\}=span\{e_{i_{0}},e_{i_{1}}\},
\]
a contradiction. Thus $A$ is not $\left(  2,2^{\aleph_{1}}\right)  $-lineable,
but it is clear that $A$ is $\left(  1,2^{\aleph_{1}}\right)  $-lineable.

\end{example}

Obviously, every known lineability/spaceability result can be investigated in
this more restrictive setting. In this paper we develop a technique to
characterize the $\left(  \alpha,\mathfrak{c}\right)  $-spaceability of
$\ell_{p}$$\setminus%
{\textstyle\bigcup_{0<q<p}}
\ell_{q}$ for all $\alpha<\mathfrak{c}$ that illustrates the technicalities
arisen by this new approach. We also propose open problems in the final section.

\section{Characterization of $\left(  \alpha,\mathfrak{c}\right)
$-spaceability in sequence spaces}

It is well known that $\ell_{p}$\bigskip$\setminus%
{\textstyle\bigcup_{0<q<p}}
\ell_{q}$ is $\mathfrak{c}$-spaceable (see \cite{Pellegrino2}). In this
section we characterize $\left(  \alpha,\mathfrak{c}\right)  $-spaceability in
this setting. For the sake of completeness, we begin by presenting a simple
explicit example of a vector inside $\ell_{p}$\bigskip$\setminus%
{\textstyle\bigcup_{0<q<p}}
\ell_{q}$ that we were not able to find in the literature. Let
\begin{equation}
\mathbb{N=}\bigcup\limits_{j=1}^{\infty}\mathbb{N}_{j}, \label{ant}%
\end{equation}
with $\mathbb{N}_{i}\cap\mathbb{N}_{j}=\emptyset$ whenever $i\neq j$ and
$card\left(  \mathbb{N}_{i}\right)  =\aleph_{0}$ for all $i$. Denote%
\[
\mathbb{N}_{j}=\{j_{1},j_{2},...\}
\]
with $j_{r}<j_{s}$ whenever $r<s$, and define%
\[
x_{j}^{(k)}=\left\{
\begin{array}
[c]{c}%
0\text{, if }j\notin\mathbb{N}_{k}\\
r^{-1/\left(  p-k^{-1}\right)  }\text{, if }j=k_{r}\in\mathbb{N}_{k}.
\end{array}
\right.
\]
Note that $\left(  x_{j}^{(k)}\right)  _{j=1}^{\infty}\in\ell_{p}$%
\bigskip$\setminus\ell_{p-\frac{1}{k}}$. Then, the sequence $\left(
y_{j}\right)  _{j=1}^{\infty}$ defined by
\[
y_{j}=\frac{x_{j}^{(k)}}{2^{k}\left\Vert \left(  x_{l}^{(k)}\right)
_{l=1}^{\infty}\right\Vert _{p}} \text{, }%
\]
where $j\in\mathbb{N}_{k},$ belongs to $\ell_{p}$\bigskip$\setminus%
{\textstyle\bigcup_{0<q<p}}
\ell_{q}.$ In fact,%
\begin{align*}
\sum\limits_{j=1}^{\infty}\left\vert y_{j}\right\vert ^{p}  &  =\sum
\limits_{k=1}^{\infty}\sum\limits_{j\in\mathbb{N}_{k}}^{\infty}\left\vert
y_{j}\right\vert ^{p}\\
&  =\sum\limits_{k=1}^{\infty}\sum\limits_{j\in\mathbb{N}_{k}}^{\infty}\left(
\frac{\left\vert x_{j}^{(k)}\right\vert }{2^{k}\left\Vert \left(  x_{l}%
^{(k)}\right)  _{l=1}^{\infty}\right\Vert _{p}}\right)  ^{p}\\
&  =\sum\limits_{k=1}^{\infty}\frac{1}{2^{kp}\left\Vert \left(  x_{l}%
^{(k)}\right)  _{l=1}^{\infty}\right\Vert _{p}^{p}}\sum\limits_{j\in
\mathbb{N}_{k}}^{\infty}\left\vert x_{j}^{(k)}\right\vert ^{p}\\
&  =\sum\limits_{k=1}^{\infty}\frac{1}{2^{kp}}<\infty,
\end{align*}
and a similar estimate shows that%
\[
\sum\limits_{j=1}^{\infty}\left\vert y_{j}\right\vert ^{q}=\infty
\]
for all $0<q<p.$

\begin{lemma}
\label{lemma_li} Let $p>0$ and $x_{1},\ldots,x_{n}$ be linearly independent
vectors of $\ell_{p}$. There exists $n_{0}\in\mathbb{N}$ such that the first
$n_{0}$ coordinates of $x_{1},\ldots,x_{n}$ form a linearly independent set of
$\mathbb{K}^{n_{0}}$.
\end{lemma}

\begin{proof}
Suppose that, for each $j\in\mathbb{N}$, there exist $a_{1j},\ldots a_{nj}%
\in\mathbb{K}$, not all equal to zero, such that
\begin{equation}
a_{1j}x_{1}+\ldots+a_{nj}x_{n}=(0,\ldots,0,\lambda_{j+1},\lambda_{j+2},\ldots)
\label{li}%
\end{equation}
It is plain that we may suppose $\Vert\alpha_{j}\Vert_{1}=1,$ where
$\alpha_{j}=(a_{1j},\ldots a_{nj})\in\mathbb{K}^{n}$, for all $j\in
\mathbb{N}.$ Since $(\alpha_{j})_{j=1}^{\infty}$ is bounded, there is a
convergent subsequence $\alpha_{j_{k}}\rightarrow\alpha=(a_{1},\ldots
,a_{n})\in\mathbb{K}^{n}$. Since $\Vert\alpha_{j}\Vert_{1}=1,$ it follows that
$\alpha\neq0.$ Then
\begin{equation}
a_{1j_{k}}x_{1}+\cdots+a_{nj_{k}}x_{n}\rightarrow a_{1}x_{1}+\cdots+a_{n}%
x_{n},\text{ when }k\rightarrow\infty. \label{li2}%
\end{equation}
On the other hand, if $\pi_{m}\colon\ell_{p}\rightarrow\mathbb{K}$ denotes the
$m$-th canonical projection, then it follows from (\ref{li}) that
\[
\pi_{m}(a_{1j_{k}}x_{1}+\cdots+a_{nj_{k}}x_{n})\rightarrow0,\text{ when
}k\rightarrow\infty,\text{ for every }m\in\mathbb{N}.
\]
Since convergence in $(\ref{li2})$ implies coordinatewise convergence, it
follows that $a_{1}x_{1}+\cdots+a_{n}x_{n}=(0,0,\ldots)$, which is a
contradiction because $x_{1},\ldots,x_{n}$ are linearly independent.
\end{proof}

\begin{theorem}
For all $p>0$ the set $\ell_{p}\setminus%
{\textstyle\bigcup_{0<q<p}}
\ell_{q}$ is $\left(  \alpha,\mathfrak{c}\right)  $-spaceable in $\ell_{p}$
if, and only if, $\alpha<\aleph_{0}$.
\end{theorem}

\begin{proof}
If $\alpha=\aleph_{0}$, then the following example shows that $\ell
_{p}\setminus{\textstyle\bigcup_{0<q<p}} \ell_{q}$ is not $\left(  \aleph
_{0},\mathfrak{c}\right)  $-spaceable:

Split
\[
\mathbb{N=}\bigcup\limits_{j=1}^{\infty}\mathbb{N}_{j}%
\]
as in (\ref{ant}). Let
\[
W:=span\{x^{(j)}:j\in\mathbb{N}\}
\]
with

\begin{itemize}
\item $x^{(1)}=\left(  1,\left(  x_{j}^{(1)}\right)  _{j>1}\right)  $ where%
\[
\left\{
\begin{array}
[c]{c}%
x_{j}^{(1)}=0,\text{ if }j\notin\mathbb{N}_{1}\\
\left(  x_{j}^{(1)}\right)  _{j\in\mathbb{N}_{1}(j>1)}\in\ell_{p}\setminus%
{\textstyle\bigcup_{0<q<p}}
\ell_{q}%
\end{array}
\right.
\]
and%
\[
\left\Vert \left(  x_{j}^{(1)}\right)  _{j\in\mathbb{N}_{1}(j>1)}\right\Vert
_{p}=2^{-1}.
\]

\item $x^{(2)}=\left(  1,\left(  x_{j}^{(2)}\right)  _{j>1}\right)  $ where%
\[
\left\{
\begin{array}
[c]{c}%
x_{j}^{(2)}=0,\text{ if }j\notin\mathbb{N}_{2}\\
\left(  x_{j}^{(2)}\right)  _{j\in\mathbb{N}_{2}(j>1)}\in\ell_{p}%
\bigskip\setminus%
{\textstyle\bigcup_{0<q<p}}
\ell_{q}%
\end{array}
\right.
\]
and%
\[
\left\Vert \left(  x_{j}^{(2)}\right)  _{j\in\mathbb{N}_{2}(j>1)}\right\Vert
_{p}=2^{-2},
\]
and so on. It is obvious that $\left(  W\setminus\{0\}\right)  \subset\ell
_{p}\setminus%
{\textstyle\bigcup_{0<q<p}}
\ell_{q}$ and $\dim\left(  W\right)  =$ $\aleph_{0}$ and there is no closed
subspace $W_{1}$ of $\ell_{p}$ such that $W\subset W_{1}$ and $\left(
W_{1}\setminus\{0\}\right)  \subset$ $\ell_{p}$$\setminus%
{\textstyle\bigcup_{0<q<p}}
\ell_{q}$, because
\[
\lim_{k\rightarrow\infty}\left\Vert \left(  x_{j}^{(k)}\right)  _{j\in
\mathbb{N}}-e_{1}\right\Vert _{p}=\lim_{k\rightarrow\infty}\left\Vert \left(
0,x_{j}^{(k)}\right)  _{j\in\mathbb{N}_{k}(j>1)}\right\Vert _{p}%
=\lim_{k\rightarrow\infty}\frac{1}{2^{k}}=0.
\]

\end{itemize}

Now we prove the theorem for $\alpha=3.$ The general case $\alpha\in
\mathbb{N}$ is easily adapted from this case.

Let $x=\left(  x_{j}\right)  _{j=1}^{\infty},$ $y=\left(  y_{j}\right)
_{j=1}^{\infty}$ and $z=(z_{j})_{j=1}^{\infty}$ be linearly independent and
\[
W:=span\{x,y,z\}\subset\left(  \ell_{p}\bigskip\setminus{\bigcup_{0<q<p}}%
\ell_{q}\right)  \cup\{0\}
\]
be a $3$-dimensional subspace of $\ell_{p}.$ Since $\left(  x_{j}\right)
_{j=1}^{\infty},\left(  y_{j}\right)  _{j=1}^{\infty}$ and $(z_{j}%
)_{j=1}^{\infty}$ are linearly independent, Lemma \ref{lemma_li} assures that
there is a natural number $n_{0}$ such that the first $n_{0}$ coordinates of
$\left(  x_{j}\right)  _{j=1}^{\infty},\left(  y_{j}\right)  _{j=1}^{\infty}$
and $(z_{j})_{j=1}^{\infty}$ form a linearly independent set of $\mathbb{K}%
^{n_{0}}$.

It is easy to pick an infinite set $\mathbb{N}_{1}:=\{\alpha_{1}<\alpha
_{2}<\cdots\}$ of positive integers such that%
\[
\left\{
\begin{array}
[c]{c}%
{\displaystyle\sum\limits_{j\notin\mathbb{N}_{1}}}
x_{j}e_{j}\in\ell_{p}\bigskip\setminus%
{\textstyle\bigcup_{0<q<p}}
\ell_{q},\\%
{\displaystyle\sum\limits_{j\notin\mathbb{N}_{1}}}
y_{j}e_{j}\in\ell_{p}\bigskip\setminus%
{\textstyle\bigcup_{0<q<p}}
\ell_{q},\\%
{\displaystyle\sum\limits_{j\notin\mathbb{N}_{1}}}
z_{j}e_{j}\in\ell_{p}\bigskip\setminus%
{\textstyle\bigcup_{0<q<p}}
\ell_{q}%
\end{array}
\right.
\]
and such that%
\begin{equation}
\max\{\left\vert x_{\alpha_{l}}\right\vert ,\left\vert y_{\alpha_{l}%
}\right\vert ,\left\vert z_{\alpha_{l}}\right\vert \}<\frac{1}{2^{l}}
\label{999}%
\end{equation}
for all $l\in\mathbb{N}$. Of course, $\mathbb{N}\setminus\mathbb{N}_{1}$ is
also infinite. Let%
\begin{equation}
\mathbb{O}:=\mathbb{N}_{1}\setminus\{1,\ldots,n_{0}\}, \label{101010}%
\end{equation}
and note that%
\begin{equation}
\left\{
\begin{array}
[c]{c}%
{\displaystyle\sum\limits_{j\notin\mathbb{O}}}
x_{j}e_{j}\in\ell_{p}\bigskip\setminus%
{\textstyle\bigcup_{0<q<p}}
\ell_{q},\\%
{\displaystyle\sum\limits_{j\notin\mathbb{O}}}
y_{j}e_{j}\in\ell_{p}\bigskip\setminus%
{\textstyle\bigcup_{0<q<p}}
\ell_{q},\\%
{\displaystyle\sum\limits_{j\notin\mathbb{O}}}
z_{j}e_{j}\in\ell_{p}\bigskip\setminus%
{\textstyle\bigcup_{0<q<p}}
\ell_{q}.
\end{array}
\right.  \label{numerar}%
\end{equation}

Split%
\[
\mathbb{O=}%
{\displaystyle\bigcup\limits_{i=1}^{\infty}}
\mathbb{O}_{i}%
\]
with $\mathbb{O}_{i}:=\{i_{1}<i_{2}<\cdots\}$ and $\mathbb{O}_{i}%
\cap\mathbb{O}_{j}=\emptyset$ for all $i\neq j.$ It is plain that
$\mathbb{N\setminus O}$ is infinite and let $f:\mathbb{N}\rightarrow
\mathbb{N}$ be injective such that%
\[
\mathbb{N\setminus O}:=\{f(1),f(2),...\}
\]
and define, for all $i$, the vector
\[
\varepsilon_{i}=\sum_{j=1}^{\infty}x_{f(j)}e_{i_{j}}\in\mathbb{K}^{\mathbb{N}%
}.
\]

It is obvious that $\varepsilon_{i}\in\ell_{p}$ for all $i.$ Define
$\widetilde{p}=1$ if $p\geq1$ and $\widetilde{p}=p$ if $0<p<1$. For
$(a_{i})_{i=1}^{\infty}\in\ell_{\widetilde{p}}$,
\[
\sum_{i=1}^{\infty}\Vert a_{i}\varepsilon_{i}\Vert_{p}^{\widetilde{p}}%
=\sum_{i=1}^{\infty}|a_{i}|^{\widetilde{p}}\Vert\varepsilon_{i}\Vert
_{p}^{\widetilde{p}}\leq\Vert x\Vert_{p}^{\widetilde{p}}\sum_{i=1}^{\infty
}\left\vert a_{i}\right\vert ^{\widetilde{p}}=\Vert x\Vert_{p}^{\widetilde{p}%
}\left\Vert (a_{i})_{i=1}^{\infty}\right\Vert _{\widetilde{p}}^{\widetilde{p}%
}<\infty.
\]
Thus $\sum_{i=1}^{\infty}\Vert a_{i}\varepsilon_{i}\Vert_{p}^{\widetilde{p}%
}<\infty$ and thus the series $\sum_{i=1}^{\infty}a_{i}\varepsilon_{i}$
converges in $\ell_{p}$. Hence, the operator%
\[
\left\{
\begin{array}
[c]{c}%
T\colon\ell_{\widetilde{p}}\longrightarrow\ell_{p},\\
~T\left(  \left(  a_{i}\right)  _{i=0}^{\infty}\right)  =a_{0}x+a_{1}%
y+a_{2}z+\sum\limits_{i=3}^{\infty}a_{i}\varepsilon_{i}%
\end{array}
\right.
\]
is well defined. It is easy to see that $T$ is linear and injective. In fact,
if
\[
T\left(  \left(  a_{i}\right)  _{i=0}^{\infty}\right)  =0
\]
then%
\[
a_{0}x+a_{1}y+a_{2}z+\sum\limits_{i=3}^{\infty}a_{i}\varepsilon_{i}=0.
\]
In particular, since the first $n_{0}$ coordinates of each $\varepsilon_{i}$
are all zero, we have%
\[
a_{0}\left(  x_{j}\right)  _{j=1}^{n_{0}}+a_{1}\left(  y_{j}\right)
_{j=1}^{n_{0}}+a_{2}\left(  z_{j}\right)  _{j=1}^{n_{0}}=0
\]
and then
\begin{equation}
a_{0}=a_{1}=a_{2}=0. \label{aq}%
\end{equation}
From (\ref{aq}) and since the supports of the vectors $\varepsilon_{i}$ are
disjoint, we have $a_{i}=0$ for all $i$.

Note that $T\left(  \left(  a_{i}\right)  _{i=0}^{\infty}\right)  \notin
\ell_{q}$ for all $0<q<p$ and $\left(  a_{i}\right)  _{i=0}^{\infty}\neq0$. In
fact, let $i_{0}$ be such that $a_{i_{0}}\neq0.$ If $a_{0}=a_{1}=a_{2}=0$ we
have%
\[
\left\Vert T\left(  \left(  a_{i}\right)  _{i=0}^{\infty}\right)  \right\Vert
_{q}=\left\Vert \sum\limits_{i=3}^{\infty}a_{i}\varepsilon_{i}\right\Vert
_{q}\geq\left\vert a_{i_{0}}\right\vert \left\Vert \varepsilon_{i_{0}%
}\right\Vert _{q}=\infty.
\]
If $a_{i}\neq0$ for some $i=0,1,2,$ since the coordinates of $\varepsilon_{i}$
belonging to $\mathbb{N\setminus O}$ are zero, we have%
\[
\left\Vert T\left(  \left(  a_{i}\right)  _{i=0}^{\infty}\right)  \right\Vert
_{q}=\left\Vert a_{0}x+a_{1}y+a_{2}z+\sum\limits_{i=3}^{\infty}a_{i}%
\varepsilon_{i}\right\Vert _{q}\geq\left\Vert \left(  a_{0}x_{j}+a_{1}%
y_{j}+a_{2}z_{j}\right)  _{j\notin\mathbb{O}}\right\Vert _{q}.
\]
Since%
\[
\infty=\left\Vert \left(  a_{0}x_{j}+a_{1}y_{j}+a_{2}z_{j}\right)
_{j\in\mathbb{N}}\right\Vert _{q}^{q}=\left\Vert \left(  a_{0}x_{j}+a_{1}%
y_{j}+a_{2}z_{j}\right)  _{j\in\mathbb{O}}\right\Vert _{q}^{q}+\left\Vert
\left(  a_{0}x_{j}+a_{1}y_{j}+a_{2}z_{j}\right)  _{j\notin\mathbb{O}%
}\right\Vert _{q}^{q}%
\]
and by $(\ref{999})$ and $(\ref{101010})$ we have%
\[
\left\Vert \left(  a_{0}x_{j}+a_{1}y_{j}+a_{2}z_{j}\right)  _{j\in\mathbb{O}%
}\right\Vert _{q}<\infty,
\]
then%
\begin{equation}
\left\Vert \left(  a_{0}x_{j}+a_{1}y_{j}+a_{2}z_{j}\right)  _{j\notin
\mathbb{O}}\right\Vert _{q}=\infty. \label{111111}%
\end{equation}
Thus $T\left(  \left(  a_{i}\right)  _{i=0}^{\infty}\right)  \notin\ell_{q}$
for all $0<q<p.$ We have thus just proved the $\left(  3,\mathfrak{c}\right)
$-lineability (and of course the $\left(  n,\mathfrak{c}\right)  $-lineability
for $n\in\mathbb{N}$ is analogous). Now let us prove the spaceability.

Of course, $\overline{T\left(  \ell_{\widetilde{p}}\right)  }$ is a closed
infinite-dimensional subspace of $\ell_{p}$. We just have to show that
\[
\overline{T\left(  \ell_{\widetilde{p}}\right)  }\mathbb{\setminus}\left\{
0\right\}  \subset\ell_{p}\bigskip\setminus%
{\textstyle\bigcup_{0<q<p}}
\ell_{q}.
\]

Let $w=\left(  w_{n}\right)  _{n=1}^{\infty}\in\overline{T\left(
\ell_{\widetilde{p}}\right)  },$ $w\neq0$. There are sequences $\left(
a_{i}^{(k)}\right)  _{i=0}^{\infty}\in\ell_{\widetilde{p}}$, $k\in\mathbb{N}$,
such that $w=\lim_{k\rightarrow\infty}T\left(  \left(  a_{i}^{(k)}\right)
_{i=0}^{\infty}\right)  $ in $\ell_{p}.$ Note that, for each $k\in\mathbb{N}$,%
\begin{align*}
T\left(  \left(  a_{i}^{(k)}\right)  _{i=0}^{\infty}\right)   &  =a_{0}%
^{(k)}x+a_{1}^{(k)}y+a_{2}^{(k)}z+\sum\limits_{i=3}^{\infty}a_{i}%
^{(k)}\varepsilon_{i}\\
&  =a_{0}^{(k)}x+a_{1}^{(k)}y+a_{2}^{(k)}z+\sum\limits_{i=3}^{\infty}%
a_{i}^{(k)}\sum_{j=1}^{\infty}x_{f(j)}e_{i_{j}}\\
&  =a_{0}^{(k)}x+a_{1}^{(k)}y+a_{2}^{(k)}z+\sum\limits_{i=3}^{\infty}%
\sum\limits_{j=1}^{\infty}a_{i}^{(k)}x_{f(j)}e_{i_{j}}.
\end{align*}

Since the coordinates of $\varepsilon_{i}$ belonging to $\mathbb{N\setminus
O}$ are zero, then the coordinates of $\sum\limits_{i=3}^{\infty}a_{i}%
^{(k)}\varepsilon_{i}$ belonging to $\mathbb{N\setminus O}$ are also zero.
Thus, the coordinates of $T\left(  \left(  a_{i}^{(k)}\right)  _{i=0}^{\infty
}\right)  $ belonging to $\mathbb{N\setminus O}$ are the respective
coordinates of $a_{0}^{(k)}x+a_{1}^{(k)}y+a_{2}^{(k)}z$ for all $k$. Hence the
limit%
\[
\lim_{k\rightarrow\infty}\left(  a_{0}^{(k)}x_{j}+a_{1}^{(k)}y_{j}+a_{2}%
^{(k)}z_{j}\right)  _{j\in\mathbb{N\setminus O}}%
\]
exists in $\ell_{p}\bigskip$. Since $\{\left(  x_{j}\right)  _{j\in
\mathbb{N\setminus O}},\left(  y_{j}\right)  _{j\in\mathbb{N\setminus O}%
},\left(  z_{j}\right)  _{j\in\mathbb{N\setminus O}}\}$ is linearly
independent (because $\{1,...,n_{0}\}\subset\mathbb{N\setminus O}$), it is
obvious that this limit can be written in an unique form as%
\[
\left(  ax_{j}+by_{j}+cz_{j}\right)  _{j\in\mathbb{N\setminus O}}.
\]
Applying a linear functional that sends $x_{j}$ in $1$ and $y_{j}$ and $z_{j}$
in $0$ we obtain that
\[
\lim_{k\rightarrow\infty}\left(  a_{0}^{(k)}\right)  =a.
\]
Analogously we obtain%
\[
\lim_{k\rightarrow\infty}\left(  a_{1}^{(k)}\right)  =b\text{ and }%
\lim_{k\rightarrow\infty}\left(  a_{2}^{(k)}\right)  =c.
\]

Thus%
\[
w_{j}=\left(  \lim_{k\rightarrow\infty}a_{0}^{(k)}\right)  x_{j}+\left(
\lim_{k\rightarrow\infty}a_{1}^{(k)}\right)  y_{j}+\left(  \lim_{k\rightarrow
\infty}a_{2}^{(k)}\right)  z_{j}%
\]
for all positive integers $j$ in $\mathbb{N\setminus O}.$ To finish the proof
consider the following cases:

\begin{itemize}
\item First case: $\lim_{k\rightarrow\infty}a_{i}^{(k)}\neq0,$ for some
$i=0,1,2.$

This case is simple because%
\[
\left\Vert w\right\Vert _{q}^{q}\geq\sum\limits_{j\in\mathbb{N\setminus O}%
}\left\vert a_{0}x_{j}+a_{1}y_{j}+a_{2}z_{j}\right\vert ^{q}\overset
{(\ref{111111})}{=}\infty
\]
and the proof is done.

\item Second case: $\lim_{k\rightarrow\infty}a_{i}^{(k)}=0$ for all $i=0,1,2.$
\end{itemize}

In this case $w_{j}=0$ for all $j$ in $\mathbb{N\setminus O}.$ Since $w\neq0,$
for all there is $r\in\mathbb{O}$ such that $w_{r}\neq0.$

Since $\mathbb{O}=\bigcup_{j=1}^{\infty}\mathbb{O}_{j}$, there are (unique)
$m,t\in\mathbb{N}$ such that $e_{m_{t}}=e_{r}$. Thus, for each $k\in
\mathbb{N}$, the $r$-th coordinate of $T\left(  \left(  a_{i}^{(k)}\right)
_{i=0}^{\infty}\right)  $ is the number $a_{0}^{(k)}x_{r}+a_{1}^{(k)}%
y_{r}+a_{2}^{(k)}z_{r}+a_{m}^{(k)}x_{f(t)}.$ So
\[
0\neq w_{r}=\lim_{k\rightarrow\infty}\left(  a_{0}^{(k)}x_{r}+a_{1}^{(k)}%
y_{r}+a_{2}^{(k)}z_{r}+a_{m}^{(k)}x_{f(t)}\right)  =\lim_{k\rightarrow\infty
}a_{m}^{(k)}x_{f(t)}=x_{f(t)}\cdot\lim_{k\rightarrow\infty}a_{m}^{(k)}.
\]
It follows that $x_{f(t)}\neq0.$ Hence $\lim_{k\rightarrow\infty}|a_{m}%
^{(k)}|=\frac{\left\vert w_{r}\right\vert }{\left\vert x_{f(t)}\right\vert
}\neq0$. For $j,k\in\mathbb{N}$, the $m_{j}$-th coordinate of $T\left(
\left(  a_{i}^{(k)}\right)  _{i=0}^{\infty}\right)  $ is%
\[
a_{0}^{(k)}x_{m_{j}}+a_{1}^{(k)}y_{m_{j}}+a_{2}^{(k)}z_{m_{j}}+a_{m}%
^{(k)}x_{f(j)}.
\]
Defining $\alpha_{m}=\frac{\left\vert w_{r}\right\vert }{\left\vert
x_{f(t)}\right\vert }$ we have
\[
\lim_{k\rightarrow\infty}\left\vert a_{0}^{(k)}x_{m_{j}}+a_{1}^{(k)}y_{m_{j}%
}+a_{2}^{(k)}z_{m_{j}}+a_{m}^{(k)}x_{f(j)}\right\vert =\lim_{k\rightarrow
\infty}\left\vert a_{m}^{(k)}x_{f(j)}\right\vert =\left\vert x_{f(j)}%
\right\vert \cdot\lim_{k\rightarrow\infty}|a_{m}^{(k)}|=\alpha_{m}\left\vert
x_{f(j)}\right\vert
\]
for every $j\in\mathbb{N}$. On the other hand, coordinatewise convergence
gives us%
\[
\lim_{k\rightarrow\infty}\left\vert a_{0}^{(k)}x_{m_{j}}+a_{1}^{(k)}y_{m_{j}%
}+a_{2}^{(k)}z_{m_{j}}+a_{m}^{(k)}x_{f(j)}\right\vert =\left\vert w_{m_{j}%
}\right\vert ,
\]
so $\left\vert w_{m_{j}}\right\vert =\alpha_{m}\left\vert x_{f(j)}\right\vert
$ for each $j\in\mathbb{N}$. Hence
\[
\left\Vert w\right\Vert _{q}^{q}=\sum\limits_{n=1}^{\infty}\left\vert
w_{n}\right\vert ^{q}\geq\sum\limits_{j=1}^{\infty}\left\vert w_{m_{j}%
}\right\vert ^{q}=\sum\limits_{j=1}^{\infty}\alpha_{m}^{q}\cdot\left\vert
x_{f(j)}\right\vert ^{q}=\alpha_{m}^{q}\cdot\left\Vert \left(  x_{f(j)}%
\right)  _{j=1}^{\infty}\right\Vert _{q}^{q}=\alpha_{m}^{q}\cdot\left\Vert
\left(  x_{j}\right)  _{j\in\mathbb{N\setminus O}}\right\Vert _{q}^{q}%
\overset{\text{(\ref{numerar})}}{=}\infty.
\]
\indent Therefore $w\notin%
{\textstyle\bigcup_{0<q<p}}
\ell_{q}\bigskip$, so $\overline{T\left(  \ell_{\widetilde{p}}\right)
}\setminus\left\{  0\right\}  \subseteq\ell_{p}\setminus%
{\textstyle\bigcup_{0<q<p}}
\ell_{q}.$
\end{proof}

\section{Some open problems}

Of course, any positive result of lineability and/or spaceability is a
potential problem to be investigated in our more general framework (for
instance, the results of \cite{Nacib, Bernal, BCFP_LAA, BCFPS_QJM,
Pellegrino2, Pellegrino1, BFPS_LAA, Teixeira}). Our feeling is that, in
general, new techniques and tools are needed to deal with these new problems.
We finish this paper by illustrating situations in which we were able just to prove $\left(  1,\mathfrak{c}\right)
$-lineability or $\left(  1,\mathfrak{c}\right)
$-spaceability.

\subsection{Non injective linear operators}

Lineability and spaceability are investigated in the framework of (non)
injective and (non) surjective functions/operators in several papers (see, for
instance, \cite{Nacib2, esp, Bernal2, Gamez2, Jimenez2} and the references therein). In this section we prove a related result in the new context
initiated in the present paper; our result provides only $\left(  1,\mathfrak{c}%
\right)  $-lineability; the general case seems to be an interesting open problem, as well as the case of (non) surjective operators.   

\begin{theorem}
Let $A:=\left\{  T\in\mathcal{L}\left(  \ell_{p};\ell_{q}\right)  :\text{
}T\text{ is non injective}\right\}  $. Then, $A$ is $\left(  1,\mathfrak{c}%
\right)  $-lineable.
\end{theorem}

\begin{proof}
In fact, let $T\in A\setminus\{0\}$ and consider $W_{1}$ the subspace
generated by $T$ . By hypothesis, there exist $x,y$ in $\ell_{p}$ with $x\neq
y$ such that
\begin{equation}
T\left(  x\right)  =T\left(  y\right)  . \label{q2}%
\end{equation}

Since $T\neq0$, there is a $z\in\ell_{p}$ such that $T\left(  z\right)  $
$\neq0$. Let $j_{0}$ be such that $\left(  Tz\right)  _{j_{0}}\neq0,$ where
$\left(  Tz\right)  _{j_{0}}$ denotes the $j_{0}$-th coordinate of the
sequence $T(z)$. Let us choose $\left(  \mathbb{N}_{k}\right)  _{k=1}^{\infty
}$ a sequence of pairwise disjoint subsets of $\mathbb{N}$ with $card\left(
\mathbb{N}_{k}\right)  =\aleph_{0} $ for all $k$, and such that $j_{0}\notin%
{\textstyle\bigcup\limits_{k=1}^{\infty}}
\mathbb{N}_{k}$.

Define, for each $k\in\mathbb{N}$, a sequence of linear operators $T_{k}%
:\ell_{p}\longrightarrow\ell_{q}$ of the form
\[
T_{k}\left(  x\right)  =\left\{
\begin{array}
[c]{c}%
\left(  T\left(  x\right)  \right)  _{j},\text{ if }j\in\mathbb{N}_{k}\\
0,\text{ otherwise.}%
\end{array}
\right.
\]
Note that $T\left(  x\right)  =T\left(  y\right)  \Rightarrow T_{k}\left(
x\right)  =T_{k}\left(  y\right)  $ $.$ We conclude that $T_{k}\in A$ for all
$k\in\mathbb{N}$. In addition, the set $\left\{  T,T_{k}:k\in\mathbb{N}%
\right\}  $ is linearly independent. In fact, let $a,a_{1},\ldots,a_{k}$ be
scalars and suppose that%
\[
aT+a_{1}T_{1}+\cdots+a_{k}T_{k}=0.
\]
We have
\[
aT\left(  z\right)  +a_{1}T_{1}\left(  z\right)  +\cdots+a_{k}T_{k}\left(
z\right)  =0.
\]
In particular
\[
a\left(  Tz\right)  _{j_{0}}+a_{1}\left(  T_{1}z\right)  _{j_{0}}+\cdots
+a_{k}\left(  T_{k}z\right)  _{j_{0}}=0.
\]
Since $j_{0}\notin%
{\textstyle\bigcup\limits_{k=1}^{\infty}}
\mathbb{N}_{k}$ , it follows that $\left(  T_{1}z\right)  _{j_{0}}=\cdots$
$=\left(  T_{k}z\right)  _{j_{0}}=0.$ So, $a\left(  Tz\right)  _{j_{0}}=0$ and
thus $a=0$. Consequently
\[
a_{1}T_{1}+\cdots+a_{k}T_{k}=0.
\]
Since $\{T_{1},...,T_{k}\}$ is linearly independent, we conclude that
$a_{1}=\cdots=a_{k}=0$.

Now consider the linear operator
\[
\Psi:\ell_{1}\longrightarrow\mathcal{L}\left(  \ell_{p};\ell_{q}\right)
\]
given by
\[
\Psi\left(  \left(  a_{k}\right)  _{k=1}^{\infty}\right)  =a_{1}T+%
{\textstyle\sum\limits_{j=2}^{\infty}}
a_{j}T_{j-1}.
\]
Note that $\Psi$ is well defined, because
\begin{align*}
\left\Vert \Psi\left(  \left(  a_{k}\right)  _{k=1}^{\infty}\right)
\right\Vert _{\mathcal{L}\left(  \ell_{p};\ell_{q}\right)  }  &  =\left\Vert
a_{1}T+%
{\textstyle\sum\limits_{j=2}^{\infty}}
a_{j}T_{j-1}\right\Vert _{\mathcal{L}\left(  \ell_{p};\ell_{q}\right)  }\\
&  \leq\left\Vert a_{1}T\right\Vert _{\mathcal{L}\left(  \ell_{p};\ell
_{q}\right)  }+%
{\textstyle\sum\limits_{j=2}^{\infty}}
\left\Vert a_{j}T_{j-1}\right\Vert _{\mathcal{L}\left(  \ell_{p};\ell
_{q}\right)  }\\
&  \leq\left\vert a_{1}\right\vert \left\Vert T\right\Vert _{\mathcal{L}%
\left(  \ell_{p};\ell_{q}\right)  }+%
{\textstyle\sum\limits_{j=1}^{\infty}}
\left\vert a_{j}\right\vert .\left\Vert T\right\Vert _{\mathcal{L}\left(
\ell_{p};\ell_{q}\right)  }\\
&  \leq%
{\textstyle\sum\limits_{j=1}^{\infty}}
\left\vert a_{j}\right\vert .\left\Vert T\right\Vert _{\mathcal{L}\left(
\ell_{p};\ell_{q}\right)  }<\infty.
\end{align*}
Moreover, it is not difficult to prove that $\Psi$ is injective. Let
$a=\left(  a_{k}\right)  _{k=1}^{\infty}\in\ell_{1},$ with $a_{k}\neq0$ for
some $k\in\mathbb{N}$ . So, from (\ref{q2}),
\begin{align*}
\Psi\left(  \left(  a_{k}\right)  _{k=1}^{\infty}\right)  \left(  x\right)
&  =a_{1}T\left(  x\right)  +a_{2}T_{1}\left(  x\right)  +a_{3}T_{2}\left(
x\right)  +\cdots\\
&  =a_{1}T\left(  y\right)  +a_{2}T_{1}\left(  y\right)  +a_{3}T_{2}\left(
y\right)  +\cdots\\
&  =\Psi\left(  \left(  a_{k}\right)  _{k=1}^{\infty}\right)  \left(
y\right)  ,
\end{align*}
and $\Psi\left(  \left(  a_{k}\right)  _{k=1}^{\infty}\right)  $ is non
injective$.$ Hence, \ $\Psi\left(  \ell_{1}\setminus\left\{  0\right\}
\right)  \subseteq A.$ Moreover,
\[
T\in W_{1}\subseteq\Psi\left(  \ell_{1}\setminus\left\{  0\right\}  \right)
\subseteq A\cup\left\{  0\right\}  .
\]
Therefore, $A$ is $\left(  1,\mathfrak{c}\right)  $-lineable.
\end{proof}

\subsection{$\left(  \alpha,\mathfrak{c}\right)  $-spaceability in $L_p$ spaces}
Recall that, for $1\leq p<\infty$, $L_{p}\left[  0,1\right]  $ denotes the classical
space of the (class of equivalence of) measurable functions $f:\left[  0,1\right]  \longrightarrow
\mathbb{K}$ equipped with the norm defined by
\[
\left\Vert f\right\Vert _{p}=\left(  \int\nolimits_{0}^{1}\left\vert
f(t)\right\vert ^{p}dt\right)  ^{\frac{1}{p}}.
\]
By mimicking our constructive example of the beginning of the previous
section, we can easily provide a simple construction of a function in $L_{p}\left[
0,1\right]  \diagdown\bigcup \limits_{q>p}L_{q}\left[  0,1\right].$

In \cite{BFPS_LAA} it was proved that $L_{p}\left[  0,1\right]  \diagdown\bigcup \limits_{q>p}L_{q}\left[  0,1\right]  $ is spaceable, but the proof does not assure that $L_{p}\left[  0,1\right]  \diagdown\bigcup \limits_{q>p}L_{q}\left[  0,1\right]  $ is $\left(  \alpha,\mathfrak{c}\right)  $-spaceable for some cardinal $\alpha>0.$ The next result shows that this is true for $\alpha=1.$ The question for a cardinal $1<\alpha<\mathfrak{c}$ remain unanswered. In the next proof, for any $X\subset[0,1]$, the characteristic function of $X$ on $[0,1]$ is denoted by $\chi_{X}$.

\begin{theorem}
\label{mt} $L_{p}\left[  0,1\right]  \diagdown%
{\textstyle\bigcup\limits_{q>p}}
L_{q}\left[  0,1\right]  $ is $\left(  1,\mathfrak{c}\right)  $-spaceable in
$L_{p}\left[  0,1\right]  .$
\end{theorem}

\begin{proof}
Let $f\in L_{p}\left[  0,1\right]  \diagdown%
{\textstyle\bigcup_{q>p}}
L_{q}\left[  0,1\right]  .$ It is obvious that
\[
\widetilde{f}=f\chi_{\lbrack0,1/2]}\text{ or }\widetilde{\widetilde{f}}=\text{
}f\chi_{\lbrack1/2,1]}%
\]
belongs to $L_{p}\left[  0,1\right]  \bigskip\diagdown%
{\textstyle\bigcup_{q>p}}
L_{q}\left[  0,1\right]  .$ Without loss of generality let us assume that
\[
\widetilde{f}\in L_{p}\left[  0,1\right]  \bigskip\diagdown%
{\textstyle\bigcup_{q>p}}
L_{q}\left[  0,1\right]  .
\]
Split $[1/2,1)$ as an infinite sequence of disjoint intervals $I_{n}%
=[c_{n},d_{n})$. Notice that, for every $n\in\mathbb{N}$ and every $x\in
I_{n}$, there is a unique $t_{x,n}\in\left[  0,1\right)  $ such that
\[
x=(1-t_{x,n})c_{n}+t_{x,n}d_{n}.
\]
Define
\[
f_{n}(x)=\left\{
\begin{array}
[c]{cl}%
\widetilde{f}(t_{x,n}) & \mathrm{if~}x\in I_{n},\\
0 & \mathrm{if~}x\notin I_{n}.
\end{array}
\right.
\]
It is simple to verify that our construction provides
\[
\left\Vert f_{n}\right\Vert _{p}<\left\Vert \widetilde{f}\right\Vert _{p},
\]
for every $n\in\mathbb{N}$ and so $f_{n}\in L_{p}\left[  0,1\right]  $.

Define $\widetilde{p}=1$ if $p\geq1$ and $\widetilde{p}=p$ if $0<p<1$. For
$(a_{i})_{i=1}^{\infty}\in\ell_{\widetilde{p}}$,
\[
\sum_{i=1}^{\infty}\Vert a_{i}f_{i}\Vert_{p}^{\widetilde{p}}=\sum
_{i=1}^{\infty}|a_{i}|^{\widetilde{p}}\Vert f_{i}\Vert_{p}^{\widetilde{p}}%
\leq\Vert\widetilde{f}\Vert_{p}^{\widetilde{p}}\sum_{i=1}^{\infty}\left\vert
a_{i}\right\vert ^{\widetilde{p}}=\Vert\widetilde{f}\Vert_{p}^{\widetilde{p}%
}\left\Vert (a_{i})_{i=1}^{\infty}\right\Vert _{\widetilde{p}}^{\widetilde{p}%
}<\infty.
\]
Thus $\sum_{i=1}^{\infty}\Vert a_{i}f_{i}\Vert_{p}^{\widetilde{p}}<\infty$ and
thus the series $\sum_{i=1}^{\infty}a_{i}f_{i}$ converges in $L_{p}\left[
0,1\right]  $. Hence, the operator
\[
T\colon\ell_{\widetilde{p}}\longrightarrow L_{p}\left[  0,1\right]
~~,~~T\left(  \left(  a_{i}\right)  _{i=0}^{\infty}\right)  =a_{0}%
f+\sum\limits_{i=1}^{\infty}a_{i}f_{i}%
\]
is well defined. It is easy to see that $T$ is linear and injective. In fact,
if
\[
T\left(  \left(  a_{i}\right)  _{i=0}^{\infty}\right)  =0
\]
then%
\[
a_{0}f+\sum\limits_{i=1}^{\infty}a_{i}f_{i}=0
\]
and choosing $x\in\lbrack0,1/2]$ we have%
\[
a_{0}f(x)=a_{0}f(x)+\sum\limits_{i=1}^{\infty}a_{i}f_{i}(x)=0.
\]
Since $f$ is non null on $[0,1/2]$ we conclude that $a_{0}=0$. Since
$\{f_{i}:i\in\mathbb{N}\}$ is linearly independent (they have disjoint
supports) we obtain $a_{i}=0$ for all $i$.

Thus $\overline{T\left(  \ell_{\widetilde{p}}\right)  }$ is a closed
infinite-dimensional subspace of $L_{p}\left[  0,1\right]  $. We just have to
show that
\[
\overline{T\left(  \ell_{\widetilde{p}}\right)  }-\left\{  0\right\}
\subseteq L_{p}\left[  0,1\right]  \diagdown%
{\textstyle\bigcup_{q>p}}
L_{q}\left[  0,1\right]  .
\]
Indeed, let $g\in\overline{T\left(  \ell_{\widetilde{p}}\right)  }%
\setminus\{0\}$. Thus $g\neq0$ a.e., that is, the set $[0,1]-A$ has null
measure, where
\[
A=\{x\in\lbrack0,1]:g(x)\neq0\}.
\]

Let us consider sequences $\left(  a_{i}^{(k)}\right)  _{i=0}^{\infty}\in
\ell_{\widetilde{p}}$ ($k\in\mathbb{N}$) such that $g=\lim_{k\rightarrow
\infty}T\left(  \left(  a_{i}^{(k)}\right)  _{i=0}^{\infty}\right)  $ in
$L_{p}[0,1].$ By the definition of $T$ we have%
\[
T\left(  \left(  a_{i}^{(k)}\right)  _{i=0}^{\infty}\right)  =a_{0}^{\left(
k\right)  }f+\sum\limits_{i=1}^{\infty}a_{i}^{\left(  k\right)  }f_{i}%
\overset{k\rightarrow\infty}{\longrightarrow}g\text{ in }L_{p}\left[
0,1\right]  .
\]
In particular%
\[
T\left(  \left(  a_{i}^{(k)}\right)  _{i=0}^{\infty}\right)  \chi
_{\lbrack0,1/2]}\overset{k\rightarrow\infty}{\longrightarrow}g\chi
_{\lbrack0,1/2]}\text{ in }L_{p}\left[  0,1\right]  .
\]
Since $f_{i}$ is null on the interval $[0,1/2]$ for all $i,$ we have%
\begin{equation}
a_{0}^{\left(  k\right)  }\widetilde{f}=T\left(  \left(  a_{i}^{(k)}\right)
_{i=0}^{\infty}\right)  \chi_{\lbrack0,1/2]}\overset{k\rightarrow\infty
}{\longrightarrow}g\chi_{\lbrack0,1/2]}\text{ in }L_{p}\left[  0,1\right]
.\label{h}%
\end{equation}
On the other hand, note that%
\begin{equation}
a_{0}^{\left(  k\right)  }\widetilde{f}\overset{k\rightarrow\infty
}{\longrightarrow}\alpha\widetilde{f}\text{ in }L_{p}\left[  0,1\right]
,\label{j}%
\end{equation}
where $\alpha=\lim_{k\rightarrow\infty}a_{0}^{\left(  k\right)  }$. By
(\ref{h}) and (\ref{j}) we have%
\[
g\chi_{\lbrack0,1/2]}=\alpha\widetilde{f}\text{ a.e.}%
\]
So, if $\alpha\neq0$%
\[
\left\Vert g\right\Vert _{q}^{q}=\int\nolimits_{0}^{1}\left\vert g\left(
t\right)  \right\vert ^{q}dt\geq\int\nolimits_{0}^{\frac{1}{2}}\left\vert
g\left(  t\right)  \right\vert ^{q}dt=\alpha^{q}\left\Vert f\chi
_{\lbrack0,1/2]}\right\Vert _{q}^{q}=\infty,
\]
proving that $g\notin L_{q}\left[  0,1\right]  $.

If $\alpha=0,$ then $g\chi_{\lbrack0,1/2]}=0$ a.e. Define%
\[
\widetilde{A}=\{x\in\lbrack1/2,1]:g(x)\neq0\}.
\]
Since $A$ has positive measure and $g\chi_{\lbrack0,1/2]}=0$ a.e., then
$\widetilde{A}$ has positive measure. Since
\[
\left\Vert \left(  a_{0}^{(k)}f+\sum_{n=1}^{\infty}a_{n}^{(k)}f_{n}-g\right)
\chi_{\lbrack1/2,1]}\right\Vert _{p}\overset{k\rightarrow\infty}%
{\longrightarrow}0,
\]
there is a subsequence
\[
\left(  \left(  a_{0}^{(k_{j})}f+\sum_{n=1}^{\infty}a_{n}^{(k_{j})}%
f_{n}\right)  \chi_{\lbrack1/2,1]}\right)  _{j=1}^{\infty}%
\]
such that
\[
\left(  a_{0}^{(k_{j})}f(x)+\sum_{n=1}^{\infty}a_{n}^{(k_{j})}f_{n}(x)\right)
\chi_{\lbrack1/2,1]}(x)\overset{j\rightarrow\infty}{\longrightarrow}%
g(x)\chi_{\lbrack1/2,1]}(x)\text{ a.e.}%
\]
Hence the set $[1/2,1]-B$, where $B=\left\{  x\in\lbrack1/2,1]:\text{ the
limit above holds}\right\}  $, has measure zero. Since,
\[
\widetilde{A}=(B\cap\widetilde{A})\cup(([1/2,1]-B)\cap\widetilde{A}),
\]
and $([1/2,1]-B)\cap\widetilde{A}$ has measure zero and $\widetilde{A}$ has
positive measure, it follows that $B\cap\widetilde{A}$ has positive measure.
Let $C=\{x\in\lbrack0,1/2]:f(x)=0\}$. Since $f\in L_{p}\left[  0,1\right]
\diagdown%
{\textstyle\bigcup_{q>p}}
L_{q}\left[  0,1\right]  $ it follows that $C$ has measure zero. By the fact
that each $f_{n}$ is the reproduction of $f$ on the interval $I_{n}$, it
follows that the set $C_{n}=\{x\in I_{n}:f_{n}(x)=0\}$ has measure zero, for
all $n\in\mathbb{N}$. Since
\[
B\cap\widetilde{A}=(B\cap\widetilde{A}\cap C_{n})\cup(B\cap\widetilde{A}%
\cap(I_{n}-C_{n})),
\]
and $B\cap\widetilde{A}\cap C_{n}$ has measure zero, then $B\cap\widetilde
{A}\cap(I_{n}-C_{n})$ has positive measure, for each $n\in\mathbb{N}.$ Fixing
$r\in\mathbb{N}$ and choosing $x_{0}\in B\cap\widetilde{A}\cap(I_{r}-C_{r})$,
with $x_{0}\neq1$, we have that $x_{0}\in I_{r}$, $f_{r}(x_{0})\neq0$,
$g(x_{0})\neq0$ and
\[
a_{0}^{(k_{j})}f(x_{0})+a_{r}^{(k_{j})}f_{r}(x_{0})=a_{0}^{(k_{j})}%
f(x_{0})+\sum_{n=1}^{\infty}a_{n}^{(k_{j})}f_{n}(x_{0})\longrightarrow
g(x_{0})\text{ when }j\rightarrow\infty.
\]
Since $\lim_{k\rightarrow\infty}a_{0}^{\left(  k\right)  }=\alpha=0$ we
obtain
\[
\lim_{j\rightarrow\infty}a_{r}^{(k_{j})}=\frac{g(x_{0})}{f_{r}(x_{0})}%
=\eta\neq0.
\]
Since
\[
f_{r}\chi_{I_{r}}(x)a_{r}^{(k_{j})}\longrightarrow g\chi_{I_{r}}(x)\text{
a.e.},\text{ when }j\rightarrow\infty,
\]
by the unicity of the limit we have
\[
g\chi_{I_{r}}=\eta f_{r}\chi_{I_{r}}\text{ a.e., }%
\]
which implies that $g\chi_{I_{r}}\notin L_{q}[0,1]$ and consequently $g\notin
L_{q}[0,1]$ (regardless of the $q>p$) finishing the proof.
\end{proof}

\end{document}